\documentclass[11pt]{article}
\usepackage[a4paper]{geometry}
\usepackage[utf8]{inputenc}
\usepackage[english]{babel}
\usepackage{amsthm,amsmath,amssymb}
\usepackage[usenames,dvipsnames,svgnames,table]{xcolor}
\usepackage[all]{xy}
\usepackage{xy,mathrsfs,pstricks,color,verbatim,numprint}
\usepackage{color,xcolor,transparent,subfigure,soul}  
\usepackage{tikz,ltablex,booktabs}
\usetikzlibrary{positioning,arrows.meta,decorations.text,calc}
\usetikzlibrary{shapes.multipart}

\usepackage{verbatim}
\input xy
\xyoption{all}

\newdir{ >}{{}*!/-5pt/\dir{>}}

\theoremstyle{definition} 
\newtheorem{theorem}{Theorem}[section]
\newtheorem{proposition}[theorem]{Proposition}

\newtheorem{remark}[theorem]{Remark}
\newtheorem{example}[theorem]{Example}
\newtheorem{definition}[theorem]{Definition}

\newcommand{\ZZ}{\mathbb{Z}}

\newcommand{\NN}{\mathbb{N}}
\DeclareMathOperator{\Hom}{Hom}

\newcommand{\imono}[1]{ \;\xymatrix{  \ar@{>->}^{#1}[r] &  \\} }
\newcommand{\iepi}[1]{ \;\xymatrix{  \ar@{->>}^{#1}[r] &  \\} }
\newcommand{\mono}{ \;\xymatrix{  \ar@{>->}[r] &  \\} }
\newcommand{\epi}{ \xymatrix{   \ar@{->>}[r] &  \\} }

\title{On the boundary algebras of the Jacobian algebras of the bordered marked surfaces}
\author{Ndongo Diouf \footnote{Postdoctoral student at Université de Sherbrooke under the supervision of Professor Thomas Brüstle}}

\begin{document}

\maketitle

\begin{abstract}
Given $\sigma$ a triangulation of bordered surface with marked points and punctures $(S, M)$, we associate an ice quiver with potential $(Q_{\sigma}, W_{\sigma}, F)$ and define the corresponding Jacobian algebra $\Gamma_{\sigma}$. We show that the boundary algebra $B(\sigma)$ of $\Gamma_{\sigma}$ depends only on the surface $(S, M)$.
\end{abstract}

\section*{Introduction}
Let $\sigma$ be a triangulation of a bordered surface with marked points $(S, M)$ without puncture. Cluster algebras $\mathcal{A}(\sigma)$ associated to the quiver $Q_{\sigma}$ were defined in \cite{CCS06}, \cite{FST08} and their quivers with potentials $(Q_{\sigma}, W_{\sigma})$ in \cite{ABC10}, \cite{Lab09}. The Jacobian algebras $J(Q_{\sigma}, W_{\sigma})$ were introduced in \cite{DWZ08}, \cite{Lab09}. The cluster category $ \mathcal{C}_{\sigma} = \mathcal{C}_{(Q_{\sigma}, W_{\sigma})}$ associated to $(Q_{\sigma}, W_{\sigma})$ is defined in \cite{Ami09} as the quotient of triangulated categories $\mathrm{Per}\Gamma /D^{b}\Gamma$, where $\Gamma = \Lambda_{(Q_{\sigma}, W_{\sigma})}$ is the corresponding Ginzburg dg-algebra \cite{Gin06} associated to $J(Q_{\sigma}, W_{\sigma})$, $\mathrm{Per}\Gamma$ is the thick subcategory of $D\Gamma$ generated by $\Gamma$ and $D^{b}\Gamma$ is the full subcategory of $D\Gamma$ of the dg-modules whose homology is of finite total dimension. 

For each bordered surface with marked points and punctures $(S, M)$, let $F$ be the set of all boundary arcs delimited by marked points. The elements of $F$ are called frozen vertices. Let  $ e = \sum_{i\in F} e_{i} $ be the sum of all primitives idempotents corresponding to the frozen vertices. We define the quotient algebra $\mathcal{P}(\sigma) = \Gamma_{\sigma} / \langle e \rangle $, where $\Gamma_{\sigma} = J(Q_{\sigma}, W_{\sigma}, F)$. The algebra $\mathcal{P}(\sigma)$ is jacobi-finite if $e \neq 0$. Let $\mathcal{G}(\mathcal{P}(\sigma)$ be the corresponding Ginzburg dg-algebra and denoted $B(\sigma)$ the boundary algebra of $\Gamma_{\sigma}$. 

Let $(R, \mathrm{m})$ be a Gorenstein local ring, that  is a commutative Noetherian local ring with finite injective dimension as an R-module. A finite $R$-module M is Cohen-Macaulay if $ \mathrm{Ext}_{A}^{i}(M, R) = 0, \:\mbox{for all $i > 0$}$. The stable category $\underline{\mathrm{CM}}(B(\sigma))$ of Cohen-Macaulay modules over $B(\sigma)$ was studied by Demonet and Luo in the case of the $(n+3)$-gon without and with one puncture $(P_{n+3,p})$, for $p \in \lbrace 0,1 \rbrace$. It is shown in \cite{DL1}, \cite{DL2} that $B(\sigma)$ is Gorenstein and is independent of the choice of the triangulation $\sigma$ of $(P_{n+3,p})$, for $p \in \lbrace 0,1 \rbrace$.
We have the following result
\begin{theorem}\cite{DL1} \cite{JKS16}\cite{DL2}\cite{BKM16}
Let  $\sigma$ be a triangulation associated to $(S, M)$. If $(S, M)= (P_{n+3,p})$, for $p\in \lbrace 0, 1 \rbrace$, then there is a triangle equivalence 
\begin{equation}
     \mathcal{C}(\mathcal{G}(\mathcal{P}(\sigma))) \cong \underline{\mathrm{CM}}(B(\sigma)) 
\end{equation} 
between the cluster category in type $\mathbb{A}_{n-3}$ (respectively in type $\mathbb{D}_{n+3}$) and the stable category of $B(\sigma)$-modules of Cohen-Macaulay \cite{Aus78}, \cite{CR90}, \cite{Sim92}, \cite{Yos90}.
\end{theorem}
It was shown by Keller-Yang in \cite{KY11} and Labardini-Fragoso in \cite{Lab09} that for any bordered surface with marked points, the cluster category $\mathcal{C}_{\sigma} = \mathcal{C}(\mathcal{G}(\mathcal{P}(\sigma)))$ depends only of the surface $(S, M)$, that is $\mathcal{C}_{\sigma}=  \mathcal{C}_{(S, M)}$. 

Our main result is the following theorem
\begin{theorem} Let  $\sigma$ be a triangulation associated to $(S, M)$. The boundary algebra $B(\sigma)$  of the frozen Jacobian algebra $\Gamma_{\sigma}$ associated to $(S,M)$ is independent of the choice of $\sigma$, up to isomorphism. 
\label{resultat principal}
\end{theorem}
In the first section of this paper, we recall some background about quiver with potential and their mutations. We show our main result in the second section and in the last section we compute some examples of $B(\sigma)$ in the cases of disk, annulus and torus. 

\section{Quiver with potential and mutations}
Let $Q = (Q_{0}, Q_{1}, s, t)$ be a finite connected quiver without loops, with set of vertices $ Q_{0} = \lbrace 1, 2, ..., n \rbrace $ and set of arrows $Q_{1}$ and $s, t$ the functions that map each arrow $\alpha$ to its starting $s(\alpha)$ and ending vertex $t(\alpha)$. We denote by $ KQ_{i} $ the $K$-vector space with basis $Q_{i}$ consisting of paths of length $i$ in $Q$, by $K Q_{i, cycl}$
the subspace of $ KQ_{i} $ spanned by all cycles in $ KQ_{i} $ and by $K\langle\langle Q\rangle\rangle = \prod_{i\geq 0}KQ_{i}$ the complete path algebra of the path algebra $ K\langle Q\rangle = \bigoplus_{i\geq 0}KQ_{i} $. An element $W$ in $ \prod_{i\geq 1}KQ_{i, cycl}$ is called potential. Two potentials $W$ and $W'$ are called cyclically equivalent if the difference $ W - W' $ is in the closure of the space generated by all differences $a_{1}...a_{n-1}a_{n} - a_{2}...a_{n}a_{1}$.  A quiver with potential is a pair $(Q,W)$ consisting of a quiver $Q$ without loops and a potential which does not have two cyclically equivalent terms. For each arrows $\alpha \in Q_{1}$, the cyclic derivative $\partial_{\alpha}$ is the continuous $K$-linear map $\xymatrix{\prod_{l\geq 1}KQ_{i, cycl} \ar[r] & K\langle\langle Q\rangle\rangle} $ acting on the cycles by
\begin{equation*}
    \partial_{\alpha}(\alpha_{1}\alpha_{2}\ldots\alpha_{k}) = \sum _{\alpha = \alpha_{i}}\alpha_{i+1}\ldots\alpha_{k}\alpha_{1}\ldots\alpha_{i-1}
\end{equation*}
The potential encodes relations on $Q$. We denoted $ \mathcal{J}(W) = \langle\langle\partial_{a}W \: \vert \:  a \in Q_{1}\rangle\rangle $ the closure of the bilateral ideal $\langle \partial_{a}W \rangle$ generated by this relations.
\begin{definition}\cite{BIRS}
An \textit{ice quiver with potential} is a triple $ (Q, W, F) $, where $ (Q, W) $ is a quiver with
potential and $ F$ is a subset of $Q_{0} $. The elements of $ F$ are called \textit{frozen vertices}. The \textit{frozen Jacobian algebra} $\Gamma$ is the quotient
\[ \Gamma = K\langle\langle Q \rangle\rangle /\mathcal{J}(W), \] where  
\[ \mathcal{J}(W, F) = \langle\langle \partial_{\alpha}W \: \vert \: \alpha \in Q_{1} : s(\alpha)\notin F \hspace{0.17cm} or \hspace{0.17cm} b(\alpha)\notin F \rangle\rangle.\] is the ideal of $K\langle\langle Q\rangle\rangle$.
\end{definition}
For example, let $F= \lbrace 2, 3 \rbrace $ and $Q$ the quiver in the Figure~\ref{carquois}.
\begin{figure}[hbt!]
    \centering
    \xymatrix{&&&&& 1\ar[rr]^{a}  & & 2\ar[dd]^{b} \\
\\
		  &&&&& 4\ar[uu]^{d} & & 3\ar[ll]^{c}}
    \caption{Quiver with potential $W = abcd$ }
    \label{carquois}
\end{figure}
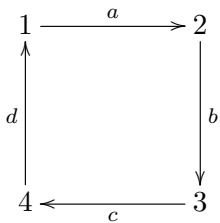
Then the Jacobian ideal is $\mathcal{J}(W) = \langle \langle bcd, dab, abc \rangle \rangle $.

Let $k \in Q_{0}$. Asume that the quiver $Q$ does not have to $2$-cycles at $k$ and no cyclic path  occurring in the expansion of $W$ starts and ends at $k$. Let $\mathrm{m}$ be the ideal generated by the arrows of $Q$. A quiver with potential is called \textit{reduced} if $\partial_{\alpha}W$ is contained in $\mathrm{m}^{2}$ for all arrow $\alpha \in Q$.

\begin{definition}\cite{DWZ08}
The \textit{mutation} of a quiver with potential $(Q, W)$ at $k$ is the new reduced quiver with potential $(\overline{Q}, \overline{W})$ = $(\tilde{Q}_{red}, \tilde{W}_{red})$ defined as follows. 
\begin{enumerate}
\item The quiver $\tilde{Q}$ is obtained from $Q$ by
\item[-] For each path $\xymatrix{i \ar[r]^\alpha & k \ar[r]^\beta & j }$, add a new arrow $\xymatrix{i \ar[r]^{[\alpha\beta]} & j}$.
\item[-] Replace each arrow $ \alpha $ with source or target $k$ with an arrow $\alpha^{\ast}$ oriented in the opposite direction.   
\item The potential $\tilde{W}$ is the sum of the two potentials $[W]$ and $\Delta_{k}$. 
\item[-] The potential $[W]$ is obtained from $W$ by replacing each composition $\alpha \beta $ by $[\alpha \beta]$, where $\alpha$ and $\beta$ are an arrows with target and source $k$ respectively. 
\item[-] The potential $\Delta_{k}$ is given by $\Delta_{k} = \displaystyle\sum_{s(\beta) = t(\alpha)=k}[\alpha\beta]\beta^{\ast}\alpha^{\ast}$
\end{enumerate}
\end{definition}
\begin{example}
The mutation at $3$ of the quiver in the Figure~\ref{carquois} corresponds to the quiver with potential $(\overline{Q}, \overline{W})$  
$$\xymatrix{1\ar[rr]^{a} & & 2\ar[ddll]^{[bc]} \\
\\
4\ar[uu]^{d}\ar[rr]_{c^{\ast}} & & 3\ar[uu]_{b\ast}} $$ $$ \overline{W} = a[bc]d + [bc]c^{\ast}b^{\ast} $$ The vertex $2$ is not eligible for mutation in $(\overline{Q}, \overline{W})$. To do this, we replace $\overline{W}$ with a cyclically equivalent potential, for example $W' = a[bc]d + c^{\ast}b^{\ast}[bc]$.  
\end{example}

\section{Jacobian algebras for triangulated surfaces}
\subsection{Marked surfaces}
Let $S$ be a connected oriented $2$-dimensional Riemann surface with boundary $\partial S$ and $M$ a finite set of \textit{marked points} in the closure of $S$ with at least one marked point on each connected component of $\partial S$. Marked points in the interior of $S$ are called \textit{punctures}. For technical reason, 
A \textit{bordered surface with marked points} is an object $(S, M)$ that is none of a sphere with $1$, a monogon with $0$ or $1$ puncture; or a digon or triangle without punctures. Up to homeomorphism, $(S, M)$ is defined by the genus $g$ of $S$, the number $b$ of boundary components, the number $c$ of marked points on the boundary and the number $p$ of punctures. Examples of a disk with $b =1$, $c = 4$ and $p = 2$, anulus with $b = 2$, $c = 5$ and $p =1$ and tore with $g =1$, $b=1$ and $c= 1$ are given in Figure \ref{exemples_de_surfaces}. 

\begin{figure}[hbt!]
    \centering
    \begin{minipage}[t]{.46\linewidth}
     \centering
    \includegraphics[ width = 6.5cm]{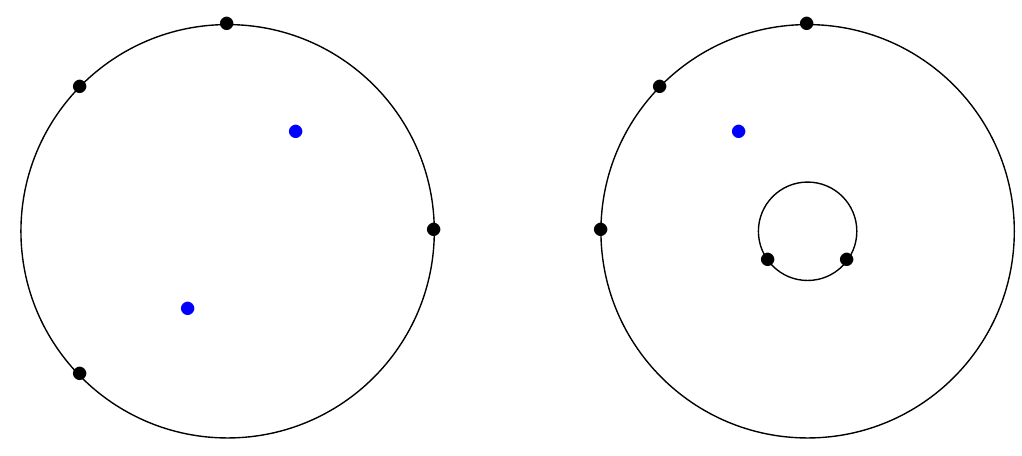}
\end{minipage}
\begin{minipage}[t]{.4\linewidth}
  \centering 
  \includegraphics[ width = 3.56cm]{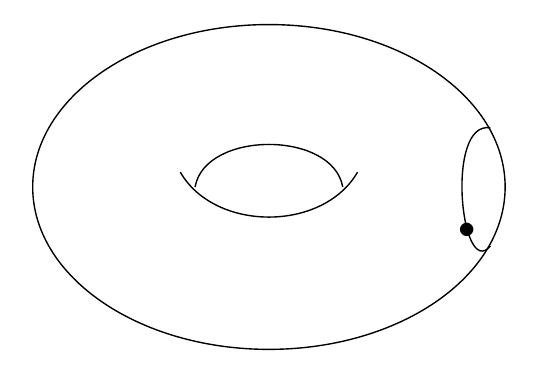}
  \end{minipage}
    \caption{Disk $D(4,2)$, anulus $A(3,2)$ and torus $T(1,0)$.}
    \label{exemples_de_surfaces}
\end{figure}
\subsection{Arcs and triangulations}
\begin{definition}\cite{FST08}
A \textit{simple arc} in $(S,M)$ is a curve in $ S $, that is the image  $\delta([0, 1])$ of continuous function $ \delta : [0, 1] \rightarrow $ S $ $ that 
\begin{itemize}
\item[$(1)$] the endpoints $\delta(0)$ and $ \delta(1) $ of $ \delta $ are in $M$; 
\item[$(2)$] except for $\delta(0)$ and $ \delta(1) $, $ \delta $ is disjoint from $M$ and from $\partial S$;
\item[$(3)$] $ \delta $ does not cross itself, except that its endpoints may coincide;
\item[$(4)$] $ \delta $ does not cut out an unpunctured monogon or an unpunctured digon
  \end{itemize}
   \label{arc simple}
The set of all simple arcs is denoted by $ A^{\circ}(S,M) $. Each arc $y \in A^{\circ}(S,M) $ is considered up to isotopy inside the class of such curves. The isotopy class of $y$ is denoted by $I(y)$. The figure \ref{exemple d'arc et non arc} showed different types of simple arcs in $(S, M)$ in the right and curves no arcs in the left.
\begin{figure}[hbt!]
     \centering
     \includegraphics[width = 7cm]{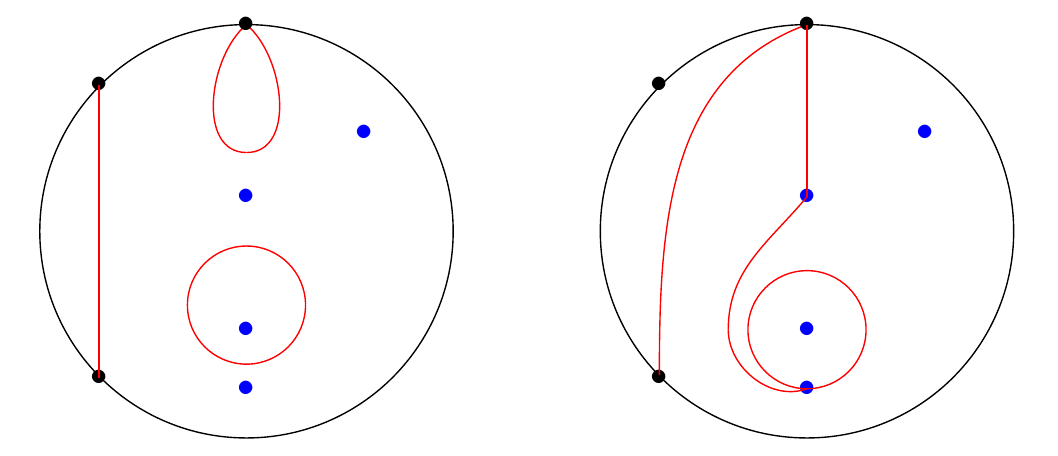}
     \caption{Simples arcs in right and curves no arcs in left.}
 \label{exemple d'arc et non arc}
 \end{figure}
\end{definition}
\begin{figure}[hbt!]
    \centering
    \includegraphics[width= 5cm]{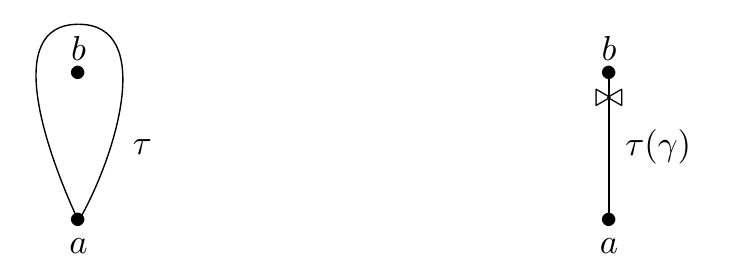}
    \caption{Representation of tagged arc} 
    \label{Represent of tagget arc}
\end{figure}
\begin{definition}\cite{FST08}
A \textit{tagged arc} $\gamma$ in $(S,M)$ is an arc in which each end has been tagged in one of two ways,
plain or notched satisfying the following conditions:
\begin{enumerate}
    \item the arc does not cut out a once-punctured monogon;
    \item an endpoint lying on the boundary $\partial_{S}$ is tagged plain; and
    \item both ends of a loop are tagged in the same way.
\end{enumerate}
Example of notched tags is showed in the figure \ref{Represent of tagget arc}. The figure \ref{Diff typ of notched targgets arcs in $(S, M)$} shows different types of notched targgets arcs in $(S, M)$.
\end{definition}
\begin{figure}[hbt!]
    \centering
    \includegraphics[width= 3.4cm]{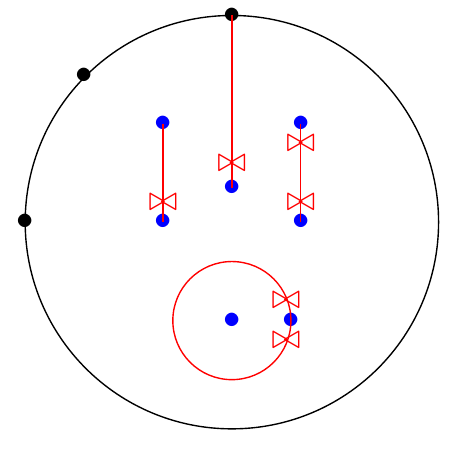}
   \caption{Notched targgets arcs in $(S, M)$.}
   \label{Diff typ of notched targgets arcs in $(S, M)$}
\end{figure} 
For any two arcs $ \gamma $, $ \rho $ in the set of tagged arcs $A^{\Join}(S,M)$, let $y$ et $p$ in $ A^{\circ}(S,M) $ the corresponding simple arcs, respectively. We define 
\begin{equation*}
    \#(y, p) = min\lbrace \text{number of crossings of $y'$ and $p'$ } \, | \,\text{ $y' \in \mathcal{I}(y)$ and $p' \in \mathcal{I}(p)$} \rbrace,
\end{equation*}
The arcs $y$ and $p$ are called \textit{compatible} if $ \#(y, p) = 0 $ in $int(S)\setminus M $.

\begin{definition}\cite{FST08}
Two tagged arcs $ \gamma $ and $ \rho $ in $A^{\Join}(S,M)$ are called \textit{compatible}  if and only if the following conditions are satisfied:
\begin{enumerate}
\item $\#(y, p) = 0$ in $int(S)\setminus M $;
\item If $\mathcal{I}(y) = \mathcal{I}(p)$, then at least one end of $ \gamma $ must
be tagged in the same way as the corresponding end of $ \rho $;
\item If $\mathcal{I}(y) \neq \mathcal{I}(p)$ and $y$ and $p$ share an endpoint $a$, then the ends of $ \gamma $ and $ \rho $ connecting to $a$ must be tagged in the same way.
\end{enumerate}
Otherwise, $ \gamma $ and $ \rho $ are called \textit{incompatible}.
\end{definition}

\begin{definition}
A maximal collection $C$ of pairwise compatible tagged arcs in $(S,M)$
is called a \textit{tagged triangulation} $T$. The collection $C$ is maximal when all arcs $ \mu $ in $ A^{\Join}(S,M)\setminus \textit{C} $ are incompatible with at east one arc in $ \textit{C}$.
\end{definition}
Every triangulation has exactly $n$ arcs given by \cite[Proposition 2.10.]{FST08} 
\begin{equation*}
    n =  6g + 3b + 3p + c - 6
\end{equation*}
The arcs of $T$ cut $(S,M)$ into a union of tagged triangles $\Delta$. Each pair of triangulations are related by a sequence of flips. Each flip at $k\in T$ replaces $k$ by a unique new arc $k'$ not isotopic such that $T' = (T \setminus \lbrace k \rbrace) \bigcup \lbrace k' \rbrace$ is a triangulation.

\subsection{Ice quivers with potential associated to (S, M)}
Let $\sigma$ be a triangulation of $(S,M)$ and $ 1\leq i\leq {\mid M\cap\partial{S}\mid} $. We number by $P_{i}$, with respect to anti-clockwise orientation, the marked points lying on one boundary component of $\partial S$. For any $ 2\leq i\leq {\mid M\cap \partial_{S}\mid} $, we index by $i$ the section of boundary $(P_{i-1}, P_{i})$ and by $1$ the section $ (P_{\mid M\cap \partial{S}\mid}, P_{1}) $. We assume that $S$ has one boundary component. Otherwise, we number the another component as above. 
We call external edges $i = (P_{i-1}, P_{i})$ and internal edges the tagged arcs $\gamma \in \sigma $. We associate to $\sigma$ the quiver with potential $ (Q'_{\sigma}, W'_{\sigma}) $ as in \cite{Lab09} \cite{DL1}
\begin{enumerate}
    \item The set of vertices $ (Q'_{\sigma})_{0} = \lbrace i \rbrace \bigcup \lbrace \gamma \in \sigma \rbrace $; 
     \item If two edges $ i$ and $j$ are sides of a common triangle $ \Delta $ of $\sigma$, there exist an \textit{internal arrow} $\xymatrix{ i \ar[r]^{\alpha_{i, j}} & j }$ in $ \Delta $ if $i$ is a predecessor of $j$ with respect to anti-clockwise orientation centered at the joint vertex.
    \end{enumerate}
We will consider below two type of cycles $\Delta_{i}$ and $\Delta_{p_{k}}$. A \textit{minimal cycle} of $ Q'_{\sigma} $ is a cycle in which no arrow appears more than once, and which encloses a part of the plane whose interior is connected and doesn't contain any arrow of $ Q'_{\sigma} $. For example, in Figure \ref{exemp quiver of P(7, 0)} $v_{3}u_{2}a_{2}$ and $u_{2}v_{2}y_{2}$ are minimal cycles and in Figure \ref{exemp quiver of P(4, 1)} $v_{1}u_{4}u_{4}a_{4}a_{1}a_{2}a_{3}a_{4}$ and $u_{4}v_{4}u_{3}v_{3}u_{2}v_{2}u_{1}v_{1}$ are not. The potential is
$$  W'_{\sigma} = \sum \Delta_{i} - \sum \Delta_{p_{k}}, \mbox{where}$$
\begin{enumerate}
    \item $\Delta_{i}$ is a \textit{clockwise cyclic triangle} consisting of three internal arrows inside of a triangle of $\sigma$ or (in the case of one-punctured digon) around one puncture;
     \item $ \Delta_{p_{k}} $ is a \textit{anti-clockwise orientation cycle}, joining tagged arcs incident to one puncture. 
\end{enumerate}
Let $ F = \lbrace i = (P_{i-1}, P_{i}) \; \vert \; 1\leq i\leq {\mid M\cap \partial_{S}\mid} \rbrace $ be the set of all external edges. 

\begin{definition}
The ice quiver with potential $ (Q_{\sigma}, W_{\sigma}, F) $ associated to the triangulation $\sigma$ of $(S, M)$ is defined as follows.
\begin{enumerate}
    \item $(Q_{\sigma})_{0} =  (Q'_{\sigma})_{0} \cup F $;
    \item For every marked point $P_{i}$ with at least one incident $\gamma \in \sigma$, there is an \textit{external arrow} $\xymatrix{(i-1) \ar[r]^{Y_{i}} & i}$. The set of arrows of $(Q_{\sigma})$ is given by 
\begin{align*}
(Q_{\sigma})_{1} = \lbrace \alpha_{i, j}: i, j \in (Q'_{\sigma})_{0} \rbrace \bigcup \lbrace Y_{i}:  1\leq i\leq {\mid M\cap\partial_{S}\mid} \rbrace;
\end{align*}
    \item The potential is $ W_{\sigma} = W'_{\sigma} - \sum \Delta_{P_{i}}, $ where $\Delta_{P_{i}} $ are anti-clockwise orientation cycles called \textit{big cycles}, formed around each marked points $P_{i}$. 
\end{enumerate}
The frozen Jacobian algebra is defined by $\Gamma_{\sigma} = K\langle\langle Q_{\sigma} \rangle\rangle / \mathcal{J}(W_{\sigma}, F)$,  where
\begin{equation}
    \mathcal{J}(W_{\sigma}, F) = \langle\langle \partial_{\alpha}W_{\sigma} \: \vert \: \alpha \in (Q_{\sigma})_{1} : s(\alpha)\notin F \hspace{0.17cm} or \hspace{0.17cm} b(\alpha)\notin F \rangle\rangle.
    \label{ideal jaco}
\end{equation}
\end{definition}

\begin{remark}
We can admit there exist an external arrow $Y_{i}$ around each marked point $P_{i}$ laying on $\partial_{S}$. In this case, the jacobian ideal corresponding is
\begin{equation}
    \mathcal{J}(W_{\sigma}, F) = \langle\langle \partial_{\alpha}W_{\sigma} \: \vert \: \alpha \in (Q_{\sigma})_{1}\setminus \lbrace Y_{i}\rbrace  \rangle\rangle.
\end{equation}
Each of these two definitions of $ \mathcal{J}(W_{\sigma}, F)$ will be used if necessarily. The corresponding two frozen Jacobian algebra are isomorphic, see \cite{DL1}, \cite{DL2}. 
\end{remark}
\begin{figure}[!htb]
    \centering
    \includegraphics[width = 6cm]{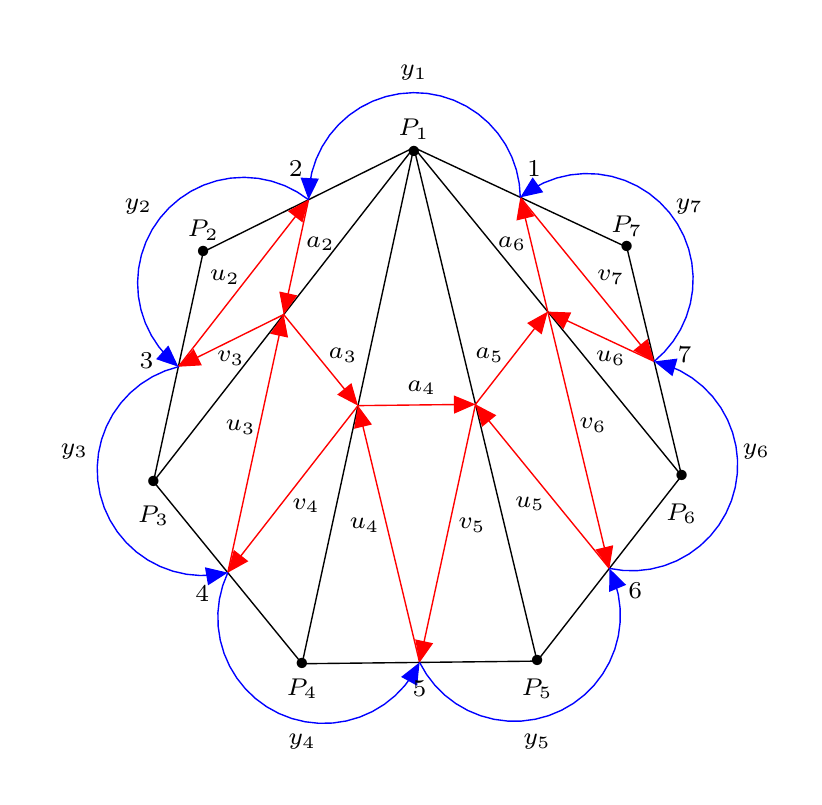}
     \caption{Quiver $Q_{\sigma}$ associated to $(P_{7, 0})$.}
\label{exemp quiver of P(7, 0)}
\end{figure}
\begin{example}
Let $Q_{\sigma}$ be the quiver associated to the heptagon without puncture $(P_{7, 0})$ in the figure \ref{exemp quiver of P(7, 0)} with potential $ W_{\sigma} = \Sigma_{i=2}^{6}v_{i+1}u_{i}a_{i} - \Sigma_{i=2}^{7}u_{i}v_{i}y_{i} - a_{2} \cdots a_{6}y_{1}$, with $ v_{2} = e_{1} $ and $ u_{7} = e_{7} $,
and set of frozen vertices $F = \lbrace 1, 2, 3, 4, 5, 6, 7 \rbrace $.
Then the Jacobian ideal is 
 \begin{equation*}
    \mathcal{J}(W_{\sigma}, F) = 
    \langle\langle 
    u_{i}a_{i} - y_{i+1}u_{i+1}, a_{i}v_{i+1} - v_{i}y_{i}, v_{i+1}u_{i} - a_{i+1} \cdots a_{n+2}y_{1}a_{2} \cdots a_{i-1}
    \rangle\rangle.
\end{equation*}
\end{example}

\begin{example}
Consider $Q_{\sigma}$ be a quiver associated to the square with one puncture $(P_{4, 1})$ in the figure \ref{exemp quiver of P(4, 1)} with potential $  W_{\sigma} = \Sigma_{i=2}^{4}v_{i+1}u_{i}a_{i} - \Sigma_{i=1}^{4}u_{i}v_{i}y_{i} - a_{1}a_{2}a_{3}a_{4} $ and set of frozen vertices $F = \lbrace 1, 2, 3, 4 \rbrace $.
\begin{figure}[!htb]
    \centering
    \includegraphics[width = 5cm]{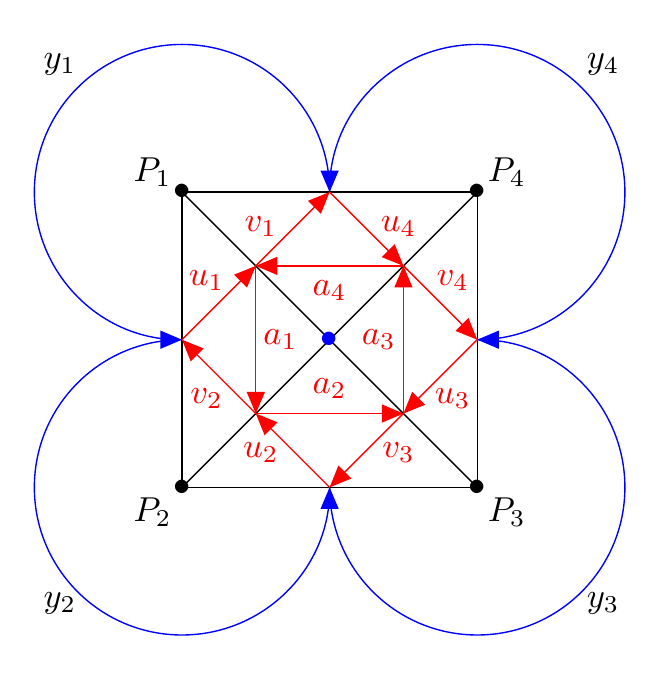}
     \caption{Quiver $Q_{\sigma}$ associated to $ (P_{4,1}) $.}
\label{exemp quiver of P(4, 1)}   
\end{figure}
Then the Jacobian ideal is 
 \begin{equation*}
    \mathcal{J}(W_{\sigma}, F) = 
    \langle\langle 
    ua - yu, av - vy, vu - a^{n+2}
    \rangle\rangle.
\end{equation*}
\end{example}

Let $ e = \sum_{i\in F} e_{i} $ be the sum of all primitives idempotents corresponding to the frozen vertices in $F$. Defined the boundary algebra $B(\sigma) = e{\Gamma_{\sigma}}e$ and the quotient $\mathcal{P}(\sigma) = \Gamma_{\sigma} / \langle e \rangle $. We have the recollement \cite{CPS88} 
\begin{equation}
    \xymatrix{ \mod{\mathcal{P}(\sigma)}  \ar[r] \ar@/^.5cm/[r]^{-\otimes_{\Gamma_{\sigma}}\mathcal{P}(\sigma)} & \mod{\Gamma_{\sigma}} \ar@/^.5cm/[r]^{-\otimes_{B(\sigma)}e\Gamma_{\sigma}} \ar@/^.5cm/[l]^{\Hom_{\Gamma_{\sigma}}(\mathcal{P}(\sigma), -)} \ar[r] & \mod{B(\sigma)} \ar@/^.5cm/[l]^{\Hom_{B(\sigma)}(\Gamma_{\sigma}e, -)} }, 
\end{equation}
where $\mathcal{P}(\sigma)$ is Jacobi-finite if $e \neq 0$. Let $\mathcal{G}(\mathcal{P}(\sigma)$ the corresponding Ginzburg dg-algebra. 
\begin{figure}[hbt!]
    \centering
    \includegraphics[width= 9.cm]{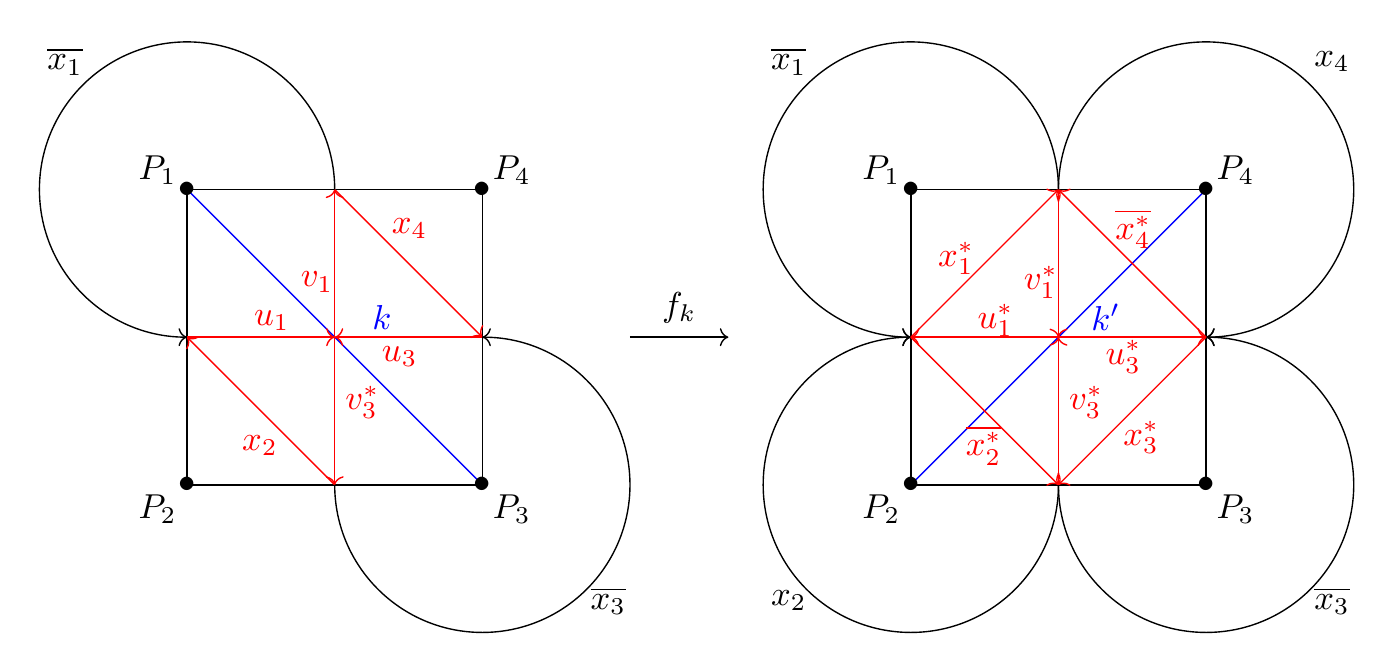}
   \caption{Local transformation in $(Q_{\sigma}, W_{\sigma}, F)$ after flip at $k$.}
\label{Chang local}
\end{figure}

\begin{theorem} \label{Th Pr Ch3} 
The boundary algebra $B(\sigma)$ of the Jacobian algebra $\Gamma_{\sigma}$ associated to the triangulation $\sigma$ of a bordered marked surface with punctures $(S,M)$ depends only on the surface, that is $B(\sigma) = B(S, M)$
\end{theorem}

\begin{proof}
Any two triangulations are obtained by sequence of mutations. So it is sufficient to show $B(\sigma) = B(\sigma')$, where $\sigma' = \mu_{k}(\sigma)$. Therefor fix $\sigma$, arc $k$ and consider mutation at $k$, see Figure \ref{Chang local}. Let $\overline{x_{1}}, \overline{x_{3}} $ the paths connecting the anti-clockwise orientated cycles around the marked points $P_{1}$ and $ P_{3}$, respectively as shown in Figure~\ref{Chang local}. Let   $W_{\sigma\setminus \lbrace k\rbrace}$ be the potentials defined by cycles in which no arrow starts or ends at $k$ and $W_{\sigma, k}$ the potential defined by all cycles which arrow starts or ends at $k$. 
Consider the path algebra $A = (1 - e_{k})kQ_{\sigma}(1 - e_{k})$, whose quiver is defined by 
\begin{equation}
    Q_{A}=\left\lbrace
	     \begin{array}{ll}
		(Q_{A})_{0} = (Q_{\sigma})_{0}\setminus \lbrace k \rbrace \\
	    (Q_{A})_{1} = ((Q_{\sigma})_{1}\setminus \lbrace u_{1}, v_{1}, u_{3}, v_{3}\rbrace)\bigcup \lbrace x_{1}, \overline{x_{2}}, x_{3}, \overline{x_{4}}\rbrace
			  \end{array}
			  \right. ,
\end{equation}
where the new arrows $x_{1}, \overline{x_{2}}, x_{3},$ and $\overline{x_{4}}$ correspond bijectively to the path $ u_{1}v_{1}, u_{1}v_{3}, u_{3}v_{3}$ and $ u_{3}v_{1} $. Let $ \alpha \in (Q_{\sigma})_{1} \bigcap \lbrace u_{1}, v_{1}, u_{3}, v_{3} \rbrace $ and $\gamma \in (Q_{\sigma})_{1}\setminus F$. The
relations generated by $\partial_{\gamma \neq \alpha}W_{\sigma\setminus \lbrace k\rbrace}$ are the same in $ Q_{A} $ as in $ Q_{\sigma} $, but $\partial_{\alpha}W_{\sigma, k}$ induced news commutative relations in $ Q_{A} $ given by: $ (a):  x_{3}x_{2} = \overline{x_{4}}\overline{x_{1}} $ and $ (a'): \overline{x_{2}}x_{2} = x_{1}\overline{x_{1}} $, $ (b): x_{4}x_{3} = \overline{x_{1}}\overline{x_{2}} $ and $ (b'): \overline{x_{1}}x_{1} = x_{4}\overline{x_{4}} $, $ (c): x_{1}x_{4} = \overline{x_{2}}\overline{x_{3}} $ and $ (c'): \overline{x_{4}}x_{4} = x_{3}\overline{x_{3}} $, $ (d): x_{2}x_{1} = \overline{x_{3}}\overline{x_{4}} $ and $(d'): x_{2}\overline{x_{2}} = \overline{x_{3}}x_{3}$. The potential $ W_{\sigma'}$ associated to $\sigma' = \mu_{k}(\sigma)$ is $W_{\sigma'} = [W_{\sigma}] + \Delta_{k} = W_{\sigma\setminus \lbrace k\rbrace} + [W_{\sigma, k}] + \Delta_{k}$,
with 
\begin{equation*}
    \Delta_{k} = [u_{1}v_{3}]v_{3}^{\ast}u_{1}^{\ast} + [u_{3}v_{1}]v_{1}^{\ast}u_{3}^{\ast} - [u_{1}v_{1}]v_{1}^{\ast}u_{1}^{\ast} - [u_{3}v_{3}]v_{3}^{\ast}u_{3}^{\ast}.    
\end{equation*}
Denoted $[u_{1}v_{1}] = x_{1}^{\ast}$, $[u_{1}v_{3}] = \overline{x_{2}^{\ast}}$, $[u_{3}v_{3}] = x_{3}^{\ast}$ and $[u_{3}v_{1}] = \overline{x_{4}^{\ast}}$ and consider the path algebra $A'=(1 - e_{k'})kQ_{\sigma'}(1 - e_{k'})$, whose quiver is defined by
\begin{equation}
    Q_{A'}=\left\lbrace 
	     \begin{array}{ll}
		(Q_{A'})_{0} = (Q_{\sigma'})_{0}\setminus \lbrace k' \rbrace \\
	    (Q_{A'})_{1} = ((Q_{\sigma'})_{1}\setminus \lbrace u_{1}^{\ast}, v_{1}^{\ast}, u_{3}^{\ast}, v_{3}^{\ast}\rbrace)\bigcup \lbrace \overline{x_{1}^{\ast}}, x_{2}^{\ast}, \overline{x_{3}^{\ast}}, x_{4}^{\ast}\rbrace
			  \end{array}
			  \right.
\end{equation}
The arrows $ \overline{ x_{1}^{\ast}}, x_{2}^{\ast}, \overline{ x_{3}^{\ast}}$, $x_{4}^{\ast}$ correspond to the paths $  v_{1}^{\ast}u_{1}^{\ast}, v_{3}^{\ast}u_{1}^{\ast}, v_{3}^{\ast}u_{3}^{\ast}$, $ v_{1}^{\ast}u_{3}^{\ast} $ in $Q_{\sigma'}$, respectively. Let $\alpha^{\ast} \in (Q_{\sigma'})_{1} \bigcap \lbrace u_{1}^{\ast}, v_{1}^{\ast}, u_{3}^{\ast}, v_{3}^{\ast}  \rbrace $, $\beta^{\ast} \in (Q_{\sigma'})_{1} \bigcap \lbrace x_{1}^{\ast}, \overline{x_{2}^{\ast}}, x_{3}^{\ast}, \overline{x_{4}^{\ast}} \rbrace $. 
We have
\begin{align*}
\partial_{\alpha \rightarrow \lbrace \alpha^{\ast},\beta^{\ast}\rbrace}([\mathcal{W}_{\sigma, k}] + \Delta_{k}) &= \partial_{\alpha^{\ast}}[\mathcal{W}_{\sigma, k}] + \partial_{\beta^{\ast}}\Delta_{k}, 
\end{align*} 
with
\begin{align*}
[\mathcal{W}_{\sigma, \ k}] &= [u_{1}v_{3}]x_{2} + [u_{3}v_{1}]x_{4} - [u_{1}v_{1}]\overline{x_{1}} + [u_{3}v_{3}]\overline{x_{3}}\\
							&= \overline{x_{2}^{\ast}}x_{2} + \overline{x_{4}^{\ast}}x_{4} -  x_{1}^{\ast}\overline{x_{1}} - x_{3}^{\ast}\overline{x_{3}}
\end{align*}
The term $\partial_{\beta^{\ast}}\Delta_{k}$ induces in $A'$ the new relations $ (e): x_{1}^{\ast}x_{4}^{\ast} = \overline{x_{2}^{\ast}}\overline{x_{3}^{\ast}} $ and $ (e'): x_{1}^{\ast}\overline{x_{1}^{\ast}} = \overline{x_{2}^{\ast}}x_{2}^{\ast} $, $ (f): \overline{x_{3}^{\ast}}\overline{x_{4}^{\ast}} = x_{2}^{\ast}x_{1}^{\ast} $ and $ (f'): x_{4}^{\ast}\overline{x_{4}^{\ast}} = \overline{x_{1}^{\ast}}x_{1}^{\ast} $, $ (g): \overline{x_{4}^{\ast}}\overline{x_{1}^{\ast}} = x_{3}^{\ast}x_{2}^{\ast} $ and $ (g'): \overline{x_{4}^{\ast}}x_{4}^{\ast} = x_{3}^{\ast}\overline{x_{3}^{\ast}} $, $ (h): \overline{x_{1}^{\ast}}\overline{x_{2}^{\ast}} = x_{4}^{\ast}x_{3}^{\ast} $, and $ (h'): x_{2}^{\ast}\overline{x_{2}^{\ast}} = \overline{x_{3}^{\ast}}x_{3}^{\ast} $. A bijective identification of the arrows $x_{i\in \lbrace 1, 3 \rbrace}$, $\overline{x}_{i\in \lbrace 2, 4 \rbrace} $ in $A$ by the new arrows $ x_{i\in \lbrace 1, 3 \rbrace}^{\ast}$, $ \overline{x}_{i\in \lbrace 2, 4 \rbrace}^{\ast} $ in $A'$ and arrows $\overline{x}_{i\in \lbrace 1, 3 \rbrace} $, $x_{i\in \lbrace 2, 4 \rbrace}$ by the arrows $\overline{x}_{i\in \lbrace 1, 3 \rbrace}^{\ast} $, $x_{i\in \lbrace 2, 4 \rbrace}^{\ast}$ gives an isomorphism of quiver algebras $\varphi: A' \to A $ 
\begin{equation*}
     \begin{array}{lrcl}
\varphi_{0} : & A'_{0} & \longrightarrow & A_{0} \\
    & e_{i} & \longmapsto & e_{i}, \: \mbox{$i\neq k$} \end{array}
\end{equation*}
and
\begin{equation*}
   \begin{array}{lrcl}
      \varphi_{1} : A'_{1} & \longrightarrow  & A_{1} \\
        & \alpha & \longmapsto & {\left\lbrace 
\begin{array}{ll}
			   \overline{x_{i}}, & \mbox{if $\alpha\in \lbrace \overline{x_{i}^{\ast}}: i= 1, 2, 3, 4 \rbrace $} \\
			   x_{i}, & \mbox{if $\alpha\in \lbrace x_{i}^{\ast}: i= 1, 2, 3, 4 \rbrace $} \\
			   \alpha, & \mbox{otherwhise}.
			  \end{array}
			  \right.}
   \end{array}
\end{equation*}
It is easy to see that $\varphi(e) = (c)$ and $\varphi(e') = (a')$, $\varphi(f) = (d)$ and $\varphi(f') = (b')$, $\varphi(g) = (a)$ et $\varphi(g') = (c')$, $\varphi(h) = (b)$ and $\varphi(h') = (d')$. Therefore, the algebras $ B(\sigma)$ are $B(\sigma')$ are isomorphic. 
\end{proof}

\section{Application}
\subsection{Case of (n+3)-gon with p-punctures.}
Let $K$ be an algebraically closed field and $G$ a finite group such that the characteristic of
$K$ does not divide the cardinality of $G$.  

If $A$ is a $K$-algebra and if $G$ acts on the right of $A$, the skew group algebra of $A$ under this action denoted  $ A \ast G  $ is defined as the left free $ A $-module with base is all elements of $G$

\begin{equation*}
 A\ast G = \lbrace \sum_{i=1}^{\vert G \vert} a_{i}g_{i} \; \vert \; a_{i}\in A, \; g_{i} \in G\rbrace = \bigoplus_{g}Ag,    
\end{equation*} 
and whose multiplication is linearly generated by 
\begin{equation*}
    (ag)(a'g') = ag(a')gg'
\end{equation*}
\begin{figure}[!htb]
    \centering
    \includegraphics[width =5cm]{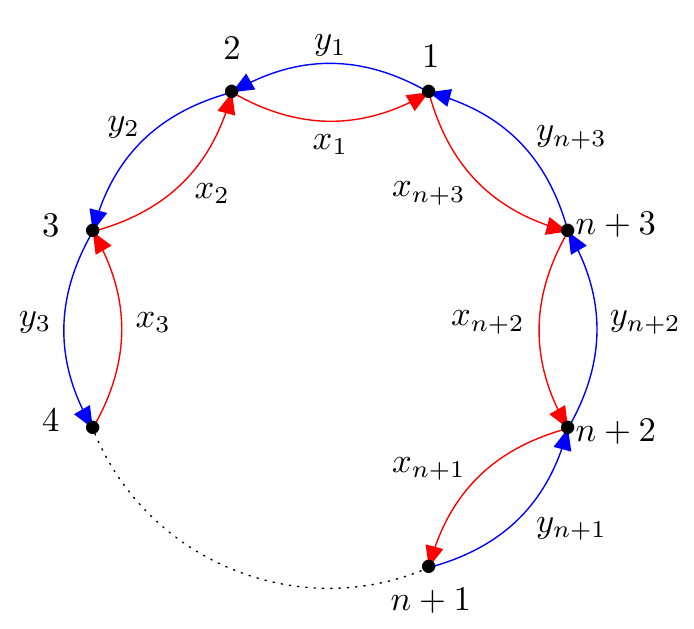}
    \caption{Quiver of $B(\sigma)$ corresponding to the quiver $\tilde{\mathbb{A}}_{n+2}$
    }
\label{cq}
\end{figure}
For example, the finite cyclic group of order $n+3\in \NN$ 
\begin{equation} 
\ZZ_{n+3} = \lbrace \bigl(\begin{smallmatrix}
 \zeta & 0 \\ 
     0 &   \zeta^{-1}
\end{smallmatrix}\bigr) \quad \vert \quad \zeta = e^{i{\frac{2k\pi}{n+3}}}, \; \mbox{k = 1, 2, ..., n+3} \rbrace 
\label{groupe cyclic de Klein}
\end{equation}  
act on $ \mathbf{C}[x, y] $ as
\begin{equation}
    \xymatrix{x \ar@{|->}[r] & \zeta x & \mbox{and} & y \ar@{|->}[r] & \zeta^{-1} y} 
    \label{action}
\end{equation}

\begin{proposition} 
The boundary algebra $B(\sigma)$ of the Jacobian algebra $\Gamma_{\sigma}$ associated to the triangulation $\sigma$ of $ (n+3)$-gone without puncture is isomorphic to the skew group algebra of $\ZZ_{n+3}$ on the Klein singularity $S_{n} = K[[x, y]] / x^2 = y^{n+1}$ of type $\mathbb{A}_{n}$.
\end{proposition}

\begin{proof}
The action of $\ZZ_{n+3}$ in (\ref{action}) on the quiver  
\begin{equation*}
    \xymatrix{1\ar@<-3pt>_{x}@(ul,dl)\ar@<-3pt>_{y}@(dr,ur)} / xy = yx
\end{equation*}
decomposes the identity $1$ into the orthogonal primitives idempotents
\begin{align}
e_{i+1} = \frac{1}{n+3}(1 + \zeta^{i}g + \zeta^{2i}g^{2} + \cdots + \zeta^{i(n+2)}g^{n+2})
\label{idempotent}
\end{align} 
corresponding to the vertices in $\tilde{\mathbb{A}}_{n+2}$. This action induces the preprojective algebra $\Pi(\tilde{\mathbb{A}}_{n+2}) \simeq K[[x,y]]\ast \ZZ_{n+3} $. Let $\sigma = \lbrace d_{1, i} = (P_{1}, P_{i})\vert 3\leq i \leq n+2 \rbrace $ be a triangulation of $(P_{n+3, 0})$ as shown in Example \ref{exemp quiver of P(7, 0)}. Define $ x_{i} = u_{i}v_{i} $, for all $ 2 \leq i \leq n+ 3 $ and $ x_{1} = a_{2}a_{3}\cdots a_{n+2} $.

The quiver of $B(\sigma) $ corresponds to that of $\Pi(\tilde{\mathbb{A}}_{n+2})$ up to relation $ \sim $. Let $  \mathcal{B} = \lbrace e_{i}, e_{i'}, \alpha _{i,j}, \beta_{k,l}\rbrace $, where $ i,j \in F $, $ i'\notin F $ and $ k $ or $ l \notin F $ be a base of $ kQ_{\sigma} $. Since all relations in $ KQ_{\sigma} $ are commutative relations, then $(  \mathcal{B} / \sim )$ is also a base of $ \Gamma_{\sigma} $ \cite[lemma 2.11.]{DL1}. Thus, $ e_{i}(\mathcal{B} / \sim )e_{j} = \lbrace \overline{e_{i'}}, \overline{\alpha_{i,j}} \rbrace = \lbrace \overline{e_{j}}, \overline{x^{m}}, \overline{y^{m}}, \overline{ua^{m}v}, \overline{(xy)^{m}} \rbrace $, for any $ i' \in F $ and $ m \geq 1 $. The potential gives the relations $u_{i}a_{i} = y_{i+1}u_{i+1}$, $a_{i}v_{i+1} = v_{i}y_{i}$ et $v_{i+1}u_{i} = a_{i+1} \cdots a_{n+2}y_{1}a_{2} \cdots a_{i-1}$. We have $x_{i}y_{i} = y_{i+1}x_{i+1}, \; \mbox{for any $ 1 \leq i \leq n+3 $}$.
Moreover,
\begin{align} 
x_{i+1}x_{i} = u_{i+1}(v_{i+1}u_{i})v_{i} &= u_{i+1}(a_{i+1} \cdots a_{n+2}y_{1}a_{2} \cdots a_{i-1})v_{i} \\
            &= (u_{i+1}a_{i+1})a_{i+2} \cdots a_{n+2}y_{1}a_{2} \cdots a_{i-2}(a_{i-1}v_{i}) \\                                    			&= y_{i+2}(u_{i+2}a_{i+2}) \cdots a_{n+2}y_{1}a_{2} \cdots (a_{i-2}v_{i-1})y_{i-1}\\
            &= \prod _{k=1}^{n+1}y_{(i+1) +k}
\end{align}       
and
\begin{align}
u_{i}a_{i}a_{i+1} \cdots a_{i+m-1}v_{i+m} &= (u_{i}a_{i})a_{i+1} \cdots a_{i+m-1}v_{i+m} \\
									 &= y_{i+1}(u_{i+1}a_{i+1})a_{i+2} \cdots a_{i+m-1}v_{i+m} \\
									&= \prod _{k=1}^{m}y_{i+k}.x_{i+m}.
\end{align}
Then, $ e( \mathcal{B} / \sim )e = \lbrace \overline{e'}, \overline{xy^{m}}, \overline{y^{m+1}} \rbrace , $ for any $ i \in F $ and $ m \in \mathbb{N}$, shut that $ \overline{xy} = \overline{yx} $ and $ \overline{x^{2}} = \overline{y^{n+1}}$. Consider the ideal $ I = \langle x^{2} - y^{n+1}\rangle $ and the $K$-linear morphism $ \varphi: \Pi(\tilde{\mathbb{A}}_{n+2}) \rightarrow B(\sigma) $ defined by $ e' \mapsto e'$, $ x\mapsto \overline{x},$ et $ y\mapsto \overline{y} .$ By construction, $ \varphi $ is surjective and $ Ker\varphi = I $. Then $\Pi(\tilde{\mathbb{A}}_{n+2}) / I \simeq B(\sigma)$. Since $ I $ is $ \ZZ_{n+3} $-stable, then $ B(\sigma) \simeq S_{n}\ast \ZZ_{n+3}$ . 
\end{proof}

\begin{proposition} 
The boundary algebra $B(\sigma)$ of the Jacobian algebra $\Gamma_{\sigma}$ associated to the triangulation $\sigma$ of a $(n+3)$-gone with $1$ puncture is isomorphic to the skew group algebra of $\ZZ_{n+3}$ on the Klein singularity $S_{2n+3} = K[[x, y]] / x^2 = y^{2n+4}$ of type $\mathbb{A}_{2n+3}$.
\end{proposition}

\begin{proof}
We use the triangulation $ \sigma = \lbrace d_{i, \bullet} = (P_{i}, P_{\bullet}) \; \vert \; 1\leq i \leq n+3 \rbrace $ of $(P_{n+3, 1})$ as showed in Figure \ref{exemp quiver of P(4, 1)}, consisting to connect each marked point $P_{i}$ on $\partial S$ to the puncture $p_{\bullet}$. The quiver of $B(\sigma)$ correspond to that of $\Pi(\tilde{\mathbb{A}}_{n+2})$. The potential 
$\mathcal{W_{\sigma}} = \Sigma_{i=1}^{n+3}v_{i+1}u_{i}a_{i} -  \Sigma_{i=1}^{n+3}u_{i}v_{i}y_{i} -a_{1}a_{2} \cdots a_{n+3} $
gives $ua = yu$, $av = vy $ and $vu = a^{n+2}$. The two first relations give $ xy = yx $. Moreover,
\begin{align}
    x^{2} = u(vu)v =ua^{n+2}v = ua^{n+1}(av) = ua^{n+1}vy = \cdots &=  ua(av)y^{n} \\ 
    &= xy^{n+2} 
\end{align}
and
\begin{align}
    ua^{m}v = (ua)a^{m-1}v = y(ua)a^{m-2}v = \cdots &=y^{m}(uv) \\ 
    &= y^{m}x 
\end{align}
Then, one base of $ B(\sigma)$ is $ \lbrace \overline{e'}, \overline{xy^{m}}, \overline{y^{m+1}} \rbrace $, for any $ i \in F $ and $ m \in \mathbb{N} $ shut that $ \overline{xy} = \overline{yx} $ and $ \overline{x^{2}} = \overline{xy^{n+2}}.$ By the same argument as before, we obtained 
\begin{equation*}
    B(\sigma) \simeq (k[[x,y]]/x^{2} = xy^{n+2})\ast \ZZ_{n+3}
\end{equation*}
Finally, the transformation $ \xymatrix{x_{i} \ar@{|->}[r] & \frac{1}{2}(x_{i} - {\displaystyle \prod_{k=1}^{n+2}y_{i+k}})} $ and $\xymatrix{y_{i} \ar@{|->}[r] & y_{i}} $ gives the transformation 
$ \xymatrix{x^2 - xy^{n+2} \ar@{|->}[r] & x^2 - y^{(2n+3)+1} } $. 
\end{proof}
\begin{figure}[!hbt]
    \centering
    \includegraphics[width=7.8cm]{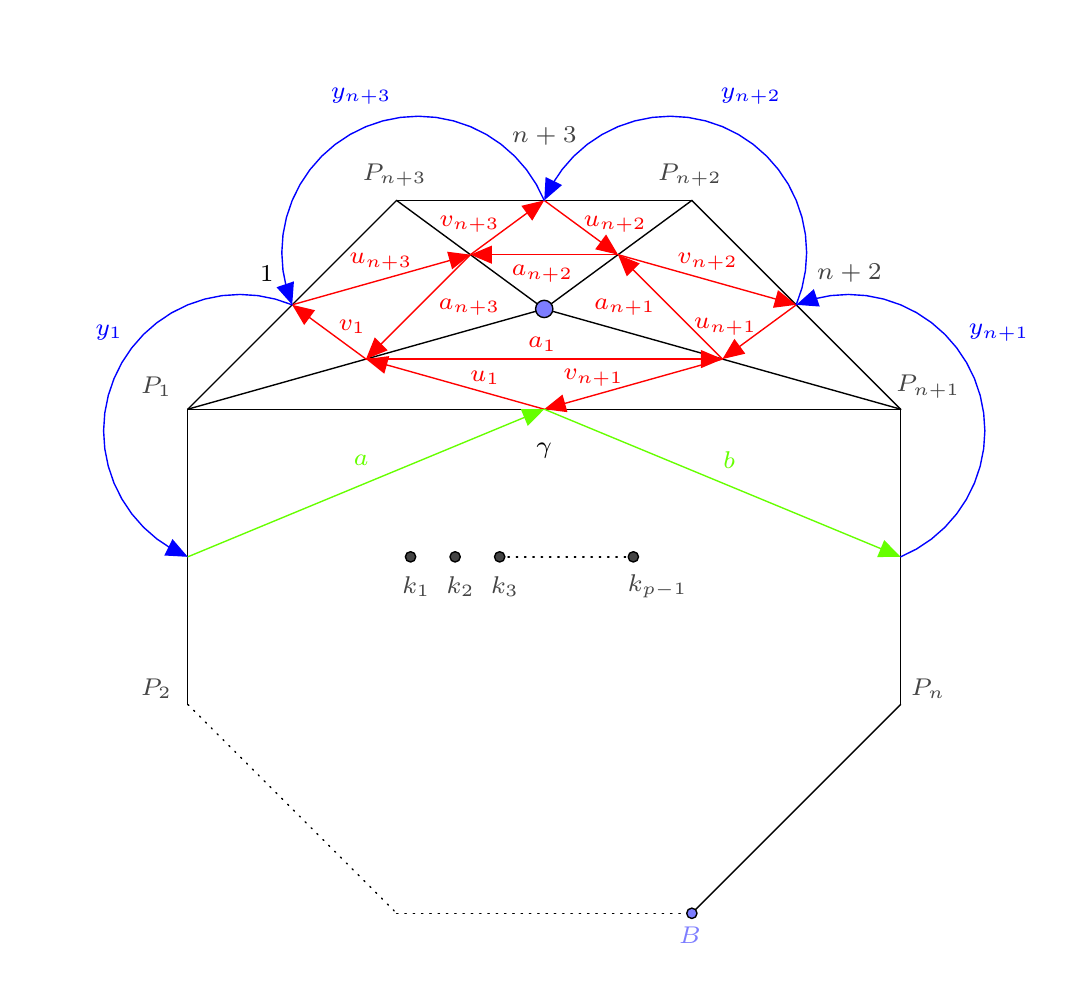}
    \caption{Triangulation $ \sigma' $ and $ \sigma'' $ in $ (P_{n+3,p}) $ }
 \label{p ponction}
\end{figure}
Denoted $\Delta_{n+3}^{p} = R_{p}\ast \ZZ_{n+3}$, where $R_{p}= K[[x, y]] / x^{2} = x^{p}y^{n+1+p}$ and $p\in \NN$. Let $ \sigma $ be a triangulation of $ (P_{n+3,p}) $ defined as in Figure~\ref{p ponction} and $ \Gamma_{\sigma}^{p} $ the corresponding Jacobian algebra, where $ \sigma'' $ and $ \sigma' $ are two triangulations associated to the polygons $ (P_{4,1}) $ and $ (P_{n+1,p-1}) $ defined by the vertices $ P_{1},P_{n+1}, P_{n+2}, P_{n+3} $ and $ P_{1},P_{2},...,P_{n+1} $, respectively. By construction, we have $ x_{i}y_{i} = x_{i+1}y_{i+1}$, for all $1\leq i\leq n+3$.

\begin{theorem} 
The boundary algebra $ B_{p}(\sigma)$ of the Jacobian algebra $ \Gamma_{\sigma}^{p} $ associated to the triangulation $\sigma$ of a $(n+3)$-gone with $p$ punctures is isomorphic to the skew group algebra $ R_{p}\ast \ZZ_{n+3}$, where 
\begin{equation*}
     \ZZ_{n+3} = \lbrace \bigl(\begin{smallmatrix}
 \zeta & 0 \\ 
     0 &   \zeta^{-1}
\end{smallmatrix}\bigr) \quad \vert \quad \zeta = e^{i{\frac{2k\pi}{n+3}}}, \mbox{where k = 1, 2, ..., n+3} \rbrace
\end{equation*}
\label{t3}
\end{theorem}

\begin{proof}
The cases $p =0$ and $p = 1$ are verified. Assume the recurrence hypothesis on $p$ is true for the rank $(p-1) $. Let $(-)_{i}$ be the variables corresponding to the paths $u_{i}v_{i}\cdots w_{i}$ and $ (P_{n+3,p}) $. Defined $(-)_{[i]}^{n+1,p-1}$ and $(-)_{[j}]^{4,1}$ those of $ (P_{n+1,p-1}) $ and $(P_{4,1})$, respectively. For any $i\neq \lbrace 1, n+1\rbrace$, we have $(-)_{[i]}^{4,1} = (-)_{i}$ in $(P_{4,1})$ and $(-)_{[i]}^{n+1,p-1} = (-)_{i}$ in $ (P_{n+1,p-1}) $. Moreover,
\begin{equation*}
    x_{1}^{n+1,p-1} = a_{1}; \; \; y_{1}^{n+1,p-1} = x_{1}y_{1}; \;\; x_{n+1}^{n+1,p-1} = b_{n+1} \; \; et \; \; y_{n+1}^{n+1,p-1} = y_{n+1}x_{n+1}^{4,1};
    \end{equation*}
\begin{equation*}
    x_{1}^{4,1} = v_{1}w_{1}; \; y_{1}^{4,1} = y_{1}y_{1}^{n+1,p-1}; \; x_{n+1}^{4,1} = u_{n+1}v_{n+1} \; et \; y_{n+1}^{n+1,p-1} = y_{n+1}^{n+1,p-1}y_{n+1}.
    \end{equation*}
In $( P_{n+1,p-1})$, the recurrence hypothesis gives, for any  $1\leq i\leq n+1$, 
\begin{align}
x_{i}x_{i-1} &= \prod_{k=1}^{p-1}x_{[i-k+1]}.\prod_{k=1}^{p-1}y_{[i-(p-k)+1]}.\prod_{k=1}^{n-1}y_{[i+k]}
\label{singularite p-1}
\end{align}
For $ i\neq \lbrace n+1, n+2, n+3, 1, 2\rbrace $: 

$(1).$ The patch $ \prod_{k=1}^{p-1}x_{[i-k+1]} $ not passed by $ x_{1}^{n+1,p-1} $, not by $ x_{n+1}^{n+1,p-1}$. Thus the path $ \prod_{k=1}^{n-1}y_{[i+k]} $ in $ B_{p-1}(\sigma'')$ passed by $ y_{n+1}^{n+1,p-1} $ and $ y_{1}^{n+1,p-1} $. But  
\begin{equation}
    (y_{n+1}^{n+1,p-1}.y_{1}^{n+1,p-1}) = y_{n+1}(x_{n+1}^{4,1}x_{1}^{4,1})y_{1} = y_{n+1}(x_{n+1}y_{n+1}y_{n+2}y_{n+3})y_{1}
\label{prod de yny1 dans Pn+1,p-1}    
\end{equation}
So, the path $ \prod_{k=1}^{n-1}y_{[i+k]} $ becomes in $ B_{p}(\sigma)$ 
\begin{equation}
    \prod_{k=1}^{n-1}y_{[i+k]}= x_{i}.\prod_{k=1}^{n+2}y_{(i-1+k)}
    \label{prod de y dans B(p-1)}
\end{equation} 
The equations (\ref{singularite p-1}) and (\ref{prod de y dans B(p-1)}) and commutativity in $ B_{p}(\sigma)$ give
\begin{align}
x_{i}x_{i-1} &= \prod_{k=1}^{p}x_{(i+1-k)}.\prod_{k=1}^{n+p+1}y_{(i-p+k)}.
\label{relat de recurr dans B(sig)}
\end{align} 

$(2).$ Or $\prod_{k=1}^{p-1}x_{[i+1-k]} $ passed by the path $ (x_{1}^{n+1,p-1}.x_{n+1}^{n+1,p-1}) $. So, $ \prod_{k=1}^{p-1}y_{[i+1-p+k]}$ and $ \prod_{k=1}^{n-1}y_{[i+k]} $ contained each other $ (y_{n+1}^{n+1,p-1}.y_{1}^{n+1,p-1})$. Then
\begin{align}
\prod_{k=1}^{p-1}x_{[i+1-k]}.\prod_{k=1}^{p-1}y_{[i+1-p+k]} &= x_{i}x_{i-1} \cdots x_{2}[(y_{2}x_{2})^{p-i}]y_{2} \cdots y_{i}\\
                &= \prod_{k=1}^{i-1}x_{(i+1-k)}(y_{2}x_{2})^{p-i}\prod_{k=p-i+1}^{p-1}y_{(i+1-p+k)} 
\label{ppppp}                
\end{align}
The equations (\ref{ppppp}) and (\ref{prod de yny1 dans Pn+1,p-1}) and commutativity give the relation (\ref{relat de recurr dans B(sig)}). Finally, 
if $ i\in\lbrace n+1, n+2, n+3, 1, 2 \rbrace $: 

$(1)$ For $ i = 2 $, we have $ x_{2}x_{1} = (x_{2}.x_{1}^{n+1,p-1}).x_{1}^{4,1}
$ in $B_{p-1}(\sigma'')$ and $B_{1}(\sigma')$. The product $(x_{2}.x_{1}^{n+1,p-1})$ verified (\ref{singularite p-1}). But, in $B_{p}(\sigma)$, $\prod_{k=1}^{n-1}y_{[2+k]}.x_{1}^{4,1}=x_{2}.\prod_{k=1}^{n+2}y_{(1+k)} $. Then, (\ref{relat de recurr dans B(sig)}) is satisfied.

$(2)$ For $ i = n+1 $, $ x_{n+1}x_{n} = x_{n+1}^{4,1}.(x_{n+1}^{n+1,p-1}x_{n})$ in $B_{p}(\sigma)$. The product $(x_{n+1}^{n+1,p-1}x_{n})$ verified (\ref{singularite p-1}). But, $ x_{n+1}^{4,1}.\prod_{k=1}^{p-1}x_{[n+2-k]} = \prod_{k=1}^{p-1}x_{(n+2-k)} $. Then (\ref{relat de recurr dans B(sig)}) is satisfied.
We also verified the others situations. The case $i=n+3$ is trivial, the case $i=1$ is symmetric to $i=n+1$ and the case $i=2$ to $i=n+2$.
\end{proof}
\subsection{Examples of annulus A(n, m) and torus T(1, 0)}
 \def\dar[#1]{\ar@<2pt>[#1]\ar@<-2pt>[#1]}
  \entrymodifiers={!!<0pt,0.7ex>+}
For $ (S, M) = A(1, 1)$ without puncture, the Jacobian algebra $\Gamma_{\sigma}$ is defined by the quiver
\begin{align*}
     \xymatrix{ && 1 \dar[dd]  \\
Q_{\sigma}: &\bullet \ar@(ul,dl)_{y} \ar[ur]  & & \bullet \ar@(ur,dr)^{\overline{y}} \ar[ul] \\
&& 1' \ar[ur] \ar[ul]
}
\end{align*}
and its boundary algebra $B_{(1,1)}(\sigma)$ by the quiver
\begin{align*}
K(1,1)=   \xymatrix{ \bullet \ar@(ul,ur)^{x}\ar@(dl,dr)_{y} \ar@/^1pc/[rr]^{r} & & \bullet \ar@(ul,ur)^{\overline{x}}\ar@(dl,dr)_{\overline{y}}\ar@/^1pc/[ll]^{t}, 
}  
\end{align*}
with relations
\begin{align*}
    xy=yx, \; xyr=r\overline{x}\overline{y}, \; r= yr\overline{y}, \; x^{2} = r\overline{y^{2}}t \\
    \overline{x}\overline{y}=\overline{y}\overline{x}, \; \overline{x}\overline{y}t=txy, \; t= \overline{y}ty, \; \overline{x^{2}} = ty^{2}r
\end{align*}
In general case $ (S, M) = A(m, n)$ without puncture, $ B_{(n,m)}(\sigma) $ is given by the quiver
\begin{align*}
   \xymatrix{
    K[[x,y]]\ast \ZZ_{n} \ar@/^1pc/[rr]^{r_{1,1}} \ar@/^2pc/[rr]^{r_{2,2}} \ar@/^5pc/[rr]^{r_{n,m}}_{\vdots} & &  K[[\overline{x},\overline{y}]]\ast \ZZ_{m}  \ar@/^1pc/[ll]^{t_{1,1}}\ar@/^2pc/[ll]^{t_{2,2}} \ar@<3pt>@/^5pc/[ll]^{t_{n,m}}_{\vdots}
    }
\end{align*}

For $(S,M)= T(1, 0)$ without puncture,
\begin{align*}
\xymatrix{ && \bullet \ar[dd]_{e} \ar@/^.5pc/[drr]^{b} \\ 
\Gamma_{\sigma} = & \bullet \ar@(ul,dl)_{y} \ar[ur]^{a}  & & & \bullet \ar@(ul,ur)^{g}\ar@(dl,dr)_{c} \ar@/^.5pc/[ull]^{d} \ar@/^.5pc/[dll]^{h}\\
&& \bullet \ar@/^.5pc/[urr]^{f} \ar[ul]^{i}
}
\end{align*}
with the relations 
\begin{align*}
    &ei = bdcefghiy &\mbox{$(a)$}\\
    &gd = cdefghiya \;\;& \mbox{$(b)$}\\ 
    &\vdots \;\;\;\;\; \vdots\; \;\;\;\;\;\vdots & \vdots\\
    &ae = yabcdefgh \;\;& \mbox{$(i)$}
\end{align*}

\newpage

%%%%%%%%%%%%%%%%%%%%%%%%%%%%%%%%%%%%%%%%%%%%%

\bibliographystyle{plain}
\bibliography{reference}
\end{document}